\documentclass[12pt]{article}

\usepackage[top=0.75in]{geometry}
\usepackage{amsmath,amsthm,amssymb, float}
\usepackage{enumerate, xcolor}
\usepackage{todonotes, comment, hyperref}

\newtheorem{thm}{Theorem}
\newtheorem{lem}[thm]{Lemma}
\newtheorem{prop}[thm]{Proposition}
\newtheorem{cor}[thm]{Corollary}
\newtheorem{conj}{Conjecture}
\newtheorem*{claim}{Claim}

\theoremstyle{definition}
\newtheorem{defi}[thm]{Definition}

\DeclareMathOperator{\mex}{mex}

\newlength{\normalparindent}
\raggedright
\begin{document}

\title{The Arithmetic-Periodicity of \textsc{cut} for $\mathcal{C}=\{1,2c\}$}
\author{Paul Ellis and Thotsaporn Aek Thanatipanonda}
\date{January 24, 2021}

\maketitle
\thispagestyle{empty}

\begin{abstract}
\textsc{cut} is a class of partition games played on a finite number of finite piles of tokens.  Each version of \textsc{cut} is specified by a cut-set $\mathcal{C}\subseteq\mathbb{N}$.  A legal move consists of selecting one of the piles and partitioning it into $d+1$ nonempty piles, where $d\in\mathcal{C}$.  No tokens are removed from the game.  It turns out that the nim-set for any $\mathcal{C}=\{1,2c\}$ with $c\geq 2$ is arithmetic-periodic, which answers an open question of \cite{par}.  The key step is to show that there is a correspondence between the nim-sets of \textsc{cut} for $\mathcal{C}=\{1,6\}$ and the nim-sets of \textsc{cut} for $\mathcal{C}=\{1,2c\}, c\geq 4$.  The result easily extends to the case of $\mathcal{C} = \{1, 2c_1, 2c_2, 2c_3, ...\}$, where $c_1,c_2, ... \geq 2$.

\end{abstract}

\section{Partition games and the game \textsc{CUT}} 
 
In this paper, we calculate the nim-sequence of a particular class of the game \textsc{cut}.  This game is part of a larger class of \textit{partition games}.  Partition games are characterized by the fact that a legal move consists of selecting one of various piles and partitioning it into some smaller piles.  In these games, no tokens are ever removed.  Two notable examples are \textsc{couples are forever} and \textsc{grundy's game}.  In the former, a player may partition any pile of size $3$ or greater into two smaller piles.  In the latter, a player may partition any pile into two smaller piles of \textit{unequal} size.  Incidentally, the nim-sequence for these games remains unknown.

For \textsc{cut}, the story is a bit different.  Here each particular game of \textsc{cut} is defined by a set of positive integers, $\mathcal{C}$, called the \emph{cut-set}.   Then a legal move for the game with cut-set $\mathcal{C}$ consists of selecting a pile and partitioning it into $d+1$ nonempty parts, for some $d\in\mathcal{C}$.  This is the same as ``cutting'' the pile $d$ times. 

Given the breadth of combinatoral game theory writ large \cite{WW1}, the authors are surprised to learn that \textsc{cut} has only been studied very recently in \cite{par}.  We will calculate the nim-sequence for \textsc{cut} when $\mathcal{C}=\{1,2c\}$ for all $c\geq 2$.  So for the games in this paper, a player may take any of the piles and split it into either two piles or into $2c+1$ piles.  In general, a combinatorial game terminates only when there are no legal moves remaining, so for these games, this is when all the tokens are in piles of size $1$.

All of the games mentioned have two variants:  \emph{normal play} and \emph{mis\'{e}re play}.  In normal play, the last player to make a legal move wins, and in misere play, the last player loses.  We only investigate the normal play variant.



The classic Sprague-Grundy Theorem says that, under normal play, each position in any of these games (or any impartial game) is equivalent to some one-pile position of \textsc{nim}.  Hence we say that if a certain position is equivalent to a nim pile of size $n$, then it has \emph{nim-value} of $n$.  Fixing a cut-set $\mathcal{C}$, we write  $\mathcal{G}(n)$ to denote the nim-value of the position of \textsc{cut} with a single pile of size $n$.  Then $\mathcal{G}(n)$ is called the  \emph{nim-sequence} of the game with cut-set $\mathcal{C}$.  

Of course, after just one move, any game of \textsc{cut} will have more than one pile.  \textsc{cut}, like \textsc{nim}, is \emph{disjunctive}, which means each move only affects one of the existing piles.  The Sprague-Grundy Theorem further tells us the nim-values only make sense if we calculate the nim-value of a position with multiple piles using the \emph{nim-sum}.  The nim-sum of $x$ and $y$, denoted $x\oplus y$ is the binary sum modulo $2$.  For example, $9 \oplus 11 = 1001_2 \oplus 1011_2 =10_2 = 2$.   
Then to calculate the nim-sequence for \textsc{cut}, or any combinatorial game, we proceed recursively: 

\begin{defi} Given a pile of size $n$, an \emph{option} $O_n$ is a sequence of piles obtainable via a single legal move on that pile.  In other words, $O_n$ is a particular partition of 
$n$ with $d+1$ parts for some $d \in \mathcal{C}$.
Formally, $O_n = (h_0,h_1, ..., h_d)$ where $1 \leq h_i$ 
and $h_0+h_1+\dots+h_d = n$.  Let $\mathcal{O}_n$ denote the set of all options $O_n$.
\end{defi}

\begin{defi}\label{nim value of an option}
The nim-value of the option $O_n = (h_0,h_1, ..., h_d)$ is
\[\mathcal{G}(O_n) = \mathcal{G}(h_0) \oplus \mathcal{G}(h_1) \oplus
\dots \oplus \mathcal{G}(h_d).\] 
\end{defi}

Finally the Sprague-Grundy Theorem tells us that the nim-value $\mathcal{G}(n)$ is given as 
\[ \mathcal{G}(n) = \mex\{ \mathcal{G}(O_n) \mid  O_n\in\mathcal{O}_n  \},\]
where $\mex(S)$ is the least natural number missing from $S$.
For example $\mex(0,3,3,1,5) = 2.$  For readers unfamilar with the  nim-sum and the $\mex$ function, we refer them to \cite{LIP}.  Finally, we note that the nim-sequence of a game encodes its winning strategy.  The positions with nim-value $0$ are called $\mathcal{P}$-positions, meaning that the previous player can win, and the other positions are called $\mathcal{N}$-positions, meaning that the next player can win.  In the latter case, the next player wins by making a move which turns the current $\mathcal{N}$-position into a $\mathcal{P}$-position.


\section{Known results and our contribution}

Let us summarize the results in \cite{par}
to give the reader some idea about the current state of affairs.   
We found that paper to be rather clear, and we recommend working through these examples as an initial exercise.  They also elaborate on some properties of nim-sum that we only mention briefly in this paper.
Table \ref{table-known results} shows the known nim-sequences of various versions of \textsc{cut}.

\begin{table}[ht]
\begin{center}
\begin{tabular}{ |c| c | c |} \hline
 Cut-set $\mathcal{C}$ & Nim sequence & Proposition in \cite{par}  \\ \hline
 $1 \not\in \mathcal{C}$ &   $(0)^c(+1)$ & 6 \\
 &  where $c$ is the smallest element in $\mathcal{C}$ &   \\ \hline
  $1\in\mathcal{C}$ and  &   $(0,1)$ &  3 \\
$\mathcal{C}$ contains only odd numbers &  &   \\ \hline
$ \{1,2,3\} \subseteq \mathcal{C}$ &   $(0)(+1)$ & 7 \\
 & (i.e. $\mathcal{G}(n) = n-1) $ & \\ \hline
 $ \mathcal{C}=\{1,3,2c\}$ &   $(0,1)^c(+2)$ &  8 \\  \hline
\end{tabular}
\end{center}
\caption{Some known results about \textsc{CUT}.}
\label{table-known results}
\end{table}

The notation $(0,1)^3(+2)$, for example, denotes the sequence $0,1,0,1,0,1,2,3,2,3,2,3,4,5,4,5,4,5,\ldots$.  In this case, we would say that the sequence is arithmetic-periodic with period $6$ and \emph{saltus} $2$.  Formally, a sequence $\mathcal{G}(n)$ is 
\textit{arithmetic-periodic} if there is a period $P$ 
and a saltus $s>0$ such that for all $n > 0$,
\[  \mathcal{G}(n+P) = \mathcal{G}(n) + s.\]

One of the cases that  \cite{par} conjectured to also have an arithmetic-periodic nim-sequence is $\mathcal{C} = \{1, 2c\}$ for any $c \geq 2$.  Most of the rest of this paper will be the proof of that.
 The proof is quite long, and we owe some gratitude to our computer friend. It helped us to make many important observations. 
In particular, we will prove:
\begin{thm}[Main Target]\label{main target}
The nim-sequence of the game \textsc{cut} with cut-set 
$\mathcal{C} =\{1, 2c  \}$ for any $c \geq 2$ is precisely
\[   (0, 1)^c(2, 3)^c, 1, 4, (5, 4)^{c-1}, (3, 2)^c(4, 5)^c(6, 7)^c  (+8)            . \]
\end{thm}
This is Open Problem $1$ in \cite{par}, though they did prove it for the cases $c=2,3,4,5$.  (Refer to Table 1, Corollary 15, and the discussion preceeding Open Problem 1.)  

Our nim-sequence is arithmetic-periodic with period $12c$ and saltus $8$.  Arithmetic-periodic nim-sequences with a power of $2$ as saltus admit a division algorithm style decomposition of the nim-values of options.
The following is essentially Lemmas 10 and 11 in \cite{par}.

\begin{prop} \label{prop1}
Suppose there exists a period $P$, a saltus $s=2^t$,  $t>1$ such that for all $n\leq P$,
\[ \mathcal{G}(n) < s, \] 
and for all $n>P$
\[ \mathcal{G}(n) = \mathcal{G}(n - P) + s.\] 
Let $h_i = k_i P + r_i\leq N$ and $h'_i = k_i P+r'_i\leq N$  
where $1 \leq r_i, r'_i \leq P$, for all $0\leq i \leq d$.  
Suppose $O_n = (h_0,h_1,\dots,h_d)$ and $O'_n = (h'_0,h'_1,\dots,h'_d)$.
Then \[  \mathcal{G}(O_n) = \mathcal{G}(O'_n) 
\;\ \text{ if and only if } \;\  \mathcal{G}((r_0,r_1,\dots,r_d)) = \mathcal{G}((r'_0,r'_1,\dots,r'_d)) .  \]
\end{prop}

\begin{proof}
By the arithmetic-periodic property,
\[ \mathcal{G}(h_i)  = \mathcal{G}(k_i P+r_i) = k_i s + \mathcal{G}(r_i) .    \]
Then 
\begin{align*}
\mathcal{G}(O_n) &=\mathcal{G}(h_0) \oplus   \mathcal{G}(h_1) \oplus \dots \oplus \mathcal{G}(h_d)\\
&=   (k_0 \oplus k_1 \oplus \dots \oplus k_d)s + 
[ \mathcal{G}(r_0) \oplus   \mathcal{G}(r_1) \oplus \dots \oplus \mathcal{G}(r_d) ]&\text{since } \mathcal{G}(r_i) < s =2^t\\
&=   (k_0 \oplus k_1 \oplus \dots \oplus k_d)s + 
\mathcal{G}((r_0,r_1,\dots,r_d))
\end{align*}
Similarly, $\mathcal{G}(O'_n)=(k_0 \oplus k_1 \oplus \dots \oplus k_d)s + 
\mathcal{G}((r'_0,r'_1,\dots, r'_d))$
\end{proof}

In the proofs that follow, we will want to find different options with the same nim-values.  Proposition \ref{prop1} allows us to restrict many calculations to the first period of $ \mathcal{G}(n)$.

\section{Patterns in the sequence. Nim-sets}\label{section-rows}

Ideally the proof would be nice and compact, but
we cannot avoid using some properties from the nim-sequence itself.
Let's write just the first period in a table of $6$ `rows.'  

\begin{table}[ht]
\begin{center}
\begin{tabular}{clclclclclc|c|c|c}
 0 & 1 & 0 & 1 & 0 & 1 & \dots & 0 & 1 \\
 2 & 3 & 2 & 3 & 2 & 3 & \dots & 2 & 3 \\
 \textbf{1} & 4 & 5 & 4 & 5 & 4 & \dots & 5 & 4 \\
 3 & 2 & 3 & 2 & 3 & 2 & \dots & 3 & 2 \\ 
 4 & 5 & 4 & 5 & 4 & 5 & \dots & 4 & 5 \\ 
 6 & 7 & 6 & 7 & 6 & 7 & \dots & 6 & 7 
\end{tabular}
\end{center}
\caption{The first $12c$ nim-values of \textsc{CUT} for $\mathcal{C}=\{1,2c\}$, $c\geq 3$}
\end{table}
Each row has $2c$ entries with alternating entries, except
that the first entry of the third row is $1$ (marked in bold), for
a total of $2c \cdot 6 =12c$ entries for each period.
The following observations will be very useful at several key stages of the proof. 

Excepting the first entry of the third row, note that the numbers in each row don't just alternate,  they do so by nim-adding $1$.  More precisely,
\begin{equation} \tag{ob1}  \label{ob1}
\mathcal{G}(n+1) = \mathcal{G}(n)\oplus 1 \;\  \text{ for }n \neq 0\bmod 2c \;\ \text{and} \;\  n \neq 4c+1 \bmod 12c
\end{equation}

The remaining observations show us how to find different partitions of the same number (i.e., two moves from the same position) which have the same nim-value.  This one is about the final entries of the rows \emph{of the first period}.
\begin{equation} \tag{ob2}    \label{ob2.1}
\mathcal{G}(2c) \oplus \mathcal{G}(12c)   = \mathcal{G}(4c) \oplus \mathcal{G}(10c) 
= \mathcal{G}(6c) \oplus \mathcal{G}(8c) = 6
\end{equation}
The rest are not restricted to the first period, and they all leverage the fundamental fact about the nim-sum that $k\oplus k=0$ for all $k$.
 \begin{equation} \tag{ob3.1} \label{ob3.1}
 \mathcal{G}(n) \oplus \mathcal{G}(n) \oplus \mathcal{G}(m+1)  \oplus \mathcal{G}(m+1)  
 = 0 = \mathcal{G}(n+1) \oplus \mathcal{G}(n+1) \oplus \mathcal{G}(m)  \oplus \mathcal{G}(m)
 \end{equation}
Here note that the nim-value of the last entry in each row of each period is the same as two entries prior: 
 \begin{equation} \tag{ob3.2} \label{ob3.2}
 \mathcal{G}(n) \oplus \mathcal{G}(n) \oplus \mathcal{G}(2ac)  
 = \mathcal{G}(n+1) \oplus \mathcal{G}(n+1) \oplus \mathcal{G}(2ac-2)
 \end{equation}
Similarly, the first entry in any but the third row of any period is the same as two entries later: 
 \begin{equation} \tag{ob3.3} \label{ob3.3}
\mathcal{G}(n+1) \oplus \mathcal{G}(n+1) \oplus \mathcal{G}(2ac+1)  
 = \mathcal{G}(n) \oplus \mathcal{G}(n) \oplus \mathcal{G}(2ac+3)  , \; a\neq 2\bmod 6
 \end{equation}
Finally, the second entry in the third row of each period is the same as two entries later:
 \begin{equation} \tag{ob3.4} \label{ob3.4}
\mathcal{G}(n+1) \oplus \mathcal{G}(n+1) \oplus \mathcal{G}(2ac+2)  
 = \mathcal{G}(n) \oplus \mathcal{G}(n) \oplus \mathcal{G}(2ac+4), \; a= 2\bmod 6
 \end{equation}



Since our version of \textsc{cut} has two distinct types of options, it makes sense to consider the effect of each one on the nim-sequence separately.

\begin{defi}\label{def-nim set}
The nim-set, $\mathcal{N}(n,p,\mathcal{C}),$ is the set of nim values
that arise from breaking $n$ tokens into $p$ piles. Formally, 
\[ \mathcal{N}(n,p,\mathcal{C}) =  \{ \mathcal{G}(O_n) \; | \;
O_n = (h_1, h_2, h_3,\dots, h_p) \text{ where }  h_1+h_2+\dots+h_p=n\} .\]
\end{defi}

Note that each nim-value in this set, $\mathcal{G}(O_n) = \mathcal{G}(h_1) \oplus
\mathcal{G}(h_2) \oplus \dots \oplus \mathcal{G}(h_p),$  is still calculated recursively according the actual rules of the game, that is, using the whole cut-set $\mathcal{C}$.  As an initial observation, note that $\mathcal{N}(n,1,\mathcal{C}) = \{\mathcal{G}_{\mathcal{C}}(n)\}$.  Furthermore, if $\mathcal{C} = \{c_1, c_2, \dots, c_p\}$, then $  \mathcal{G}_{\mathcal{C}}(n) = \mex\{ \cup_{i=1}^p \;\  \mathcal{N}(n, c_i+1 ,\mathcal{C})   \}$.

\begin{table}[h!]
\begin{center}
\begin{tabular}{c | c l c l c l c l c l c  }
 $n$ & 1 & 2 & 3 & 4 & 5 & 6 & 7 & 8 & 9 & 10 & 11 \\ \hline
 $\mathcal{N}(n,2,\{1,6\})$& - & \{0\} & \{1\} & \{0\} & \{1\} & \{0\} & \{1\} & \{0,2\} 
 & \{1,3\} & \{0,2\} & \{1,3\} \\
 $\mathcal{N}(n,7,\{1,6\})$& - & -&-&-&-&-&\{0\} & \{1\} & \{0\} & \{1\} & \{0\}  \\
  $\mathcal{G}_{\{1,6\}}(n)$& 0 & 1&0&1&0&1&2 & 3 & 2 & 3 & 2 
\end{tabular}
\begin{tabular}{c | cl c l c l c l c l c ccc c  c }
 $n$ & 12 & 13 & 14 & 15 & 16 & 17 & 18 & 19 \\ \hline
 $\mathcal{N}(n,2,\{1,6\})$& \{0,2\} & \{3\} & \{0,1,2\} & \{0,1,3,4\} & \{0,1,2,5\} & \{0,1,3,4\} & \{0,1,2,5\}
 & \{0,1,4\} \\
 $\mathcal{N}(n,7,\{1,6\})$& \{1\} & \{0,2\} & \{1,3\} & \{0,2\} & \{1,3\} & \{0,2\} & \{1,3\} & \{0,1,2\} \\
  $\mathcal{G}_{\{1,6\}}(n)$& 3 & 1&4&5&4&5&4 & 3
  \end{tabular}
\end{center}
\label{table-nimsets}
\caption{The first $19$ nim-sets of \textsc{CUT} for $\mathcal{C}=\{1,6\}$}
\end{table}

We will use $c=3$ as a motivating example in the next section, and as the base case for our eventual induction.  Thus let $\mathcal{C}=\{1,6\}$ and so $p = 2$ or $7$.  Table \ref{table-nimsets} shows the initial terms in the nim-sequence, decomposed into their nim-sets.

\section{Entering and exiting partitions}
When we further analyze the sequence of nim-sets for $p=2$, we find that it is made up of alternating subsequences.  See Table \ref{table-nimsets, decomposed}.

\begin{table}[!h]
\begin{tabular}{c | c c  c  c  c  c c c c c c c c c c}
 $n$ & 1 & 2 & 3 & 4 & 5 & 6 & 7 & 8 & 9 & 10 & 11 & 12 & 13 & 14 & 15\\ \hline
 $\mathcal{N}(n,2,\{1,6\})$& - & 0 & 1 & 0 & 1 & 0 & 1 & 0 & 1 & 0 & 1 & 0&&&\\
& & &  &  &  & &  & 2 & 3 & 2 & 3 & 2 & 3& 2 & 3\\
 &&&&&&&&&&&&&& 1 & 0\\
 &&&&&&&&&&&&&& 0 & 1\\
 &&&&&&&&&&&&&&  & 4\\
 \\
 $n$ & 16 & 17 & 18 & 19 & 20 & 21 & 22 & 23 & 24 & 25 & 26 & 27 & 28 & 29 & 30 \\ \hline
$\mathcal{N}(n,2,\{1,6\})$&&&&&&&&&&&&&&&\\
 & 2 & 3 & 2 &&&&&&&&&&&&\\
& 1 & 0 & 1 & 0 & &&&&&&&&&&\\
 & 0 & 1 &0 & 1 & 0 & 1 & 0 & 1 & 0 & &&&&&\\
& 5 & 4 & 5 & 4 & 5 & 4 & 5 & 4 & 5 & &&&&&\\
 &&&&& 3 & 2 & 3 & 2 & 3 & 2 & 3 & 2 & 3 & 2 & 3\\
&&&&& & 6 & 7 & 6 & 7 & 6 & 7 & 6 & 7 & 6 & 7\\
&&&&&&&&&&& 0 &&&&\\
&&&&&&&&&&& 1 & 0 & 1 & 0 & 1\\
&&&&&&&&&&& 4 &5&4&5&4\\
&&&&&&&&&&&&&0&1&0\\
\end{tabular}
\label{table-nimsets, decomposed}
\caption{The first $30$ nim-sets of \textsc{CUT} for $\mathcal{C}=\{1,6\}$, decomposed into alternating subsequences.}
\end{table}

This is quite striking!  There are underlying subsequences which alternate between $a$ and $a\oplus 1$.  Furthermore, these subsequences enter and exit the sequence at very particular partitions.  Here are the partitions where these subsequences begin and end: 
\begin{itemize}
\item The first subsequence begins at $O_2 = (1,1)$ and ends at $O_{12} = (6,6)$
\item The second subsequence begins at $O_8 = (7,1)$ and ends at $O_{18} = (12,6)$
\item The third subsequence begins at $O_{14} = (13,1)$ and ends at $O_{19} = (13,6)$
\item The fourth subsequence begins at $O_{14} = (7,7)$ and ends at $O_{24} = (12,12)$
\item The fifth subsequence begins at $O_{15} = (14,1)$ and ends at $O_{24} = (18,6)$
\item The sixth subsequence begins at $O_{20} = (13,7)$ or $O_{20} = (19,1)$ and ends at $O_{30} = (24,6)$
\item The seventh subsequence begins at $O_{21} = (14,7)$ and ends at $O_{30} = (18,12)$
\item The eighth subsequence begins \textbf{and} ends at $O_{26} = (13,13)$
\item The ninth subsequence begins at $O_{26} = (19,7)$ \ldots
\item The tenth subsequence begins at $O_{26} = (25,1)$ \ldots
\item The eleventh subsequence begins at $O_{28} = (14,14)$ \ldots
\end{itemize}
Continuing in this fashion, we find that each of these subsequences begins at a partition whose parts are all either $1\bmod 2c$ or $(4c+2)\bmod 12c$, and ends at a partition whose parts are all either $0\bmod 2c$ or $(4c+1)\bmod 12c$.  This motivates the following definition.
\begin{defi}  Fix a cut-set $\mathcal{C}=\{1,2c\}$.  Call a number $n>0$ an \emph{innumber} if $n=1\bmod 2c$ or $n=(4c+2)\bmod 12c$.  Then call an option $O_n=\{h_0,\ldots,h_d\}$ an \emph{entering partition} if each $h_i$ is an innumber.
Similarly, call a number $n>0$ an \emph{outnumber} if $n=0\bmod 2c$ or $n=(4c+1)\bmod 12c$.  Then call an option $O_n=\{h_0,\ldots,h_d\}$ an \emph{exiting partition} if each $h_i$ is an outnumber.  Finally, call an option $O_n=\{h_0,\ldots,h_d\}$ an \emph{intermediate partition} if it is neither an entering nor an exiting partition.
\end{defi}

Notice that the innumbers and outnumbers are the beginnings and ends of the rows from the table in Section \ref{section-rows}.  In fact, if we consider $(4c+1)\bmod 12c$ and then $(4c+2)\bmod 12c$ through $6c\bmod 12c$ as two separate rows, this is precise.

Different partitions of the same number can give the same nim-value.  This happens frequntly. For example, the first subsequence arises from either of the following sequences of partitions:
\[ (1,1),(2,1),(2,2),(3,2),(3,3),(4,3),(4,4),(5,4),(5,5),(6,5),(6,6) \] 
\[ (1,1),(2,1),(3,1),(4,1),(5,1),(6,1),(6,2),(6,3),(6,4),(6,5),(6,6)\]
More generally, if a partition is not entering, then its nim-value can be obtained recursively. In fact, this is the case for any $p\geq 2$:
\begin{lem}\label{rem1}Assume Theorem \ref{main target} holds up to $n$.
Suppose the option $O_{n+1} = (h_1,\dots, h_p)$ is not an entering partition.  Then there is an 
option $O'_n$ so that $\mathcal{G}(O_{n+1})= \mathcal{G}(O'_n) \oplus 1$.  Similarly, if $O_n= (j_1,\dots, j_p)$ is not an exiting partition, then there is an option $O'_{n+1}$ so that $\mathcal{G}(O'_{n+1})= \mathcal{G}(O_n) \oplus 1$.
\end{lem}
\begin{proof}
Choose $i$  where  $h_i$ is not an innumber.  Then $h_i-1$ is not an outnumber. Let $O'_n = (h_1,\ldots,h_i-1,\dots,h_p)$. 
\begin{align*}
\mathcal{G}(O_{n+1}) 
&= \mathcal{G}(h_1) \oplus \dots \oplus \mathcal{G}(h_i) \oplus \dots \oplus \mathcal{G}(h_p)& \\
& = \mathcal{G}(h_1) \oplus \dots \oplus [\mathcal{G}(h_i-1) \oplus 1] \oplus \dots \oplus \mathcal{G}(h_p) &\eqref{ob1}\\
& = \mathcal{G}(O'_n) \oplus 1&
\end{align*}
In the second case, choose $i$ so that $j_i$ is not an outnumber.  Then define $O'_{n+1} = (j_1,\ldots,j_i+1,\dots,j_p)$, and apply the same calculation.
\end{proof}

Extending this calculation, the nim-value of every non-entering partition is either the same as the greatest entering partition less than or equal to it, or off by $\oplus 1$.  To be precise, we have the following definition and corollary.

\begin{defi}\label{definition of floor}
Suppose $O_n=(h_1,\ldots,h_p)$ is a partition of $n$.  For each $1\leq i\leq p$, let $\left\lfloor h_i\right\rfloor$ be the greatest innumber less than or equal to $h_i$.  Then call $\left\lfloor O_n\right\rfloor=(\left\lfloor h_1\right\rfloor,\ldots,\left\lfloor h_p\right\rfloor)$ the \emph{floor} of $O_n$.  Similarly, let $\left\lceil h_i\right\rceil$ be the least outnumber greater than or equal to $h_i$, and call $\left\lceil O_n\right\rceil=(\left\lceil h_1\right\rceil,\ldots,\left\lceil h_p\right\rceil)$ the \emph{ceiling} of $O_n$.  Write $O_n\sim O'_{n'}$ if $\left\lfloor O_n\right\rfloor=\left\lfloor O'_{n'}\right\rfloor$.  For technical reasons, in the case that $\left\lfloor h_1\right\rfloor=\left\lfloor h_2\right\rfloor=\left\lfloor h_3\right\rfloor = 2 \bmod 2c$, we set $\left\lfloor (h_1,h_2,h_3)\right\rfloor=(\left\lfloor h_1\right\rfloor-1,\left\lfloor h_2\right\rfloor-1,\left\lfloor h_3\right\rfloor)$.
\end{defi}

Note that $\sim$ defines an equivalence relation on all partitions, and that the floor and ceiling partitions are always entering and exiting partitions, respectively.  Then we have the following.  To check the exceptional case, note that $$\mathcal{G}(12cq_1+4c+2,12cq_2+4c+2,12cq_3+4c+2)=\mathcal{G}(12cq_1+4c+1,12cq_2+4c+1,12cq_3+4c+2).$$
\begin{cor}\label{height calculation}Assume Theorem \ref{main target} holds up to $\max\{n,n'\}$.
Suppose $O_n\sim O'_{n'}$. If $n= n' \bmod 2$, then $\mathcal{G}(O_n)=\mathcal{G}(O'_{n'})$.  Otherwise, $\mathcal{G}(O_n)=\mathcal{G}(O'_{n'})\oplus 1$.  

Conversely, suppose $O_n$ is a partition of $n$, $\left\lfloor O_n\right\rfloor$ is a partition of $a$, and $\left\lceil O_n\right\rceil$ is a partition of $b$.  If $a\leq n'\leq b$, and $n'= n\bmod 2$, then there is a partition $O'_{n'}$ of $n'$ with the same number of parts as $O_n$ satisfying $\mathcal{G}(O'_{n'})=\mathcal{G}(O_n)$.
\end{cor}

Using only Lemma \ref{rem1}, it might seem that the sixth subsequence $3,4,3,4,\ldots$ which begins at $O_{20}=(13,7)$ would end at $O_{25}=(13,12)$.  But we saw above that it ends at some $O_{30}$.  This is because the exiting partition $O_{25} = (13,12)$ has the same nim-value as, for example, the non-exiting partition $O_{25} = (19,6)$, and thus Lemma \ref{rem1} implies that the subsequence continues on to $O_{30} = (24,6)$.  The next result shows us that, for $p\geq 4$, this replacement is always possible. 
\begin{lem}\label{one}Assume Theorem \ref{main target} holds up to $n$.
Suppose $p \geq 4$ and $\mathcal{C} = \{1,2c\}, c\geq 2.$  Then each exiting partition in $\mathcal{N}(n,p,\mathcal{C})$ has the same nim-value as some non-exiting partition in $\mathcal{N}(n,p,\mathcal{C})$.
\end{lem}
So for $p\geq 4$, once one of these subsequences starts, it never ends.  Formally:
\begin{cor}\label{stick}Assume Theorem \ref{main target} holds up to $n$.
For $p \geq 4, c \geq 2$, if $v \in \mathcal{N}(n,p,\{1, 2c\})$, then $v \oplus 1 \in \mathcal{N}(n+1,p,\{1, 2c\})$.
\end{cor}
  It may seem overly general that we write $p\geq 4$ rather than $p=2c+1$, since $\mathcal{C}=\{1,2c\}$.  However, the proof of the main result will use an induction that involves all values of $p\geq 1$.

\begin{proof}[proof of Lemma \ref{one}]  Let $O_n=(h_1,h_2,h_3,h_4,\ldots,h_p)$ be an exiting partition.  We will find a non-exiting partition $O'_n=(h'_1,h'_2,h'_3,h'_4,\ldots,h'_p)$ with the same nim-value.  In fact, we will only alter the remainders of $h_1,h_2,h_3,h_4$ modulo $12c$, so Proposition \ref{prop1} assures us that we only need to check that the nim-sum of the nim-values of these remainders remains unaltered.
The seven outnumbers less than or equal to $12c$ are precisely $2c, 4c, 4c+1, 6c, 8c, 10c, 12c$.

\textbf{Case 1:} $h_1, h_2, h_3, h_4$ are all congruent to $0 \bmod 2c$. 

If there is a repeated value, then apply 
\eqref{ob3.2} to get a non-exiting partition
with the same nim-value. If not,
then by the pigeonhole principle, one of the pairs $\{2c,12c\}$, $\{4c,10c\}$, $\{6c,8c\}$ must be present, so we can apply \eqref{ob2.1} to get a repeated value.
 
\textbf{Case 2:} $h_1 = 4c+1$ and  $h_2, h_3, h_4 = 0 \bmod 2c$.

Again, if there is a repeated value, we can apply \eqref{ob3.2}, so assume not.

If one of the values $4c, 6c$ or $12c$ are present, we can adjust as follows to abtain a non-exiting partition:
\[ \mathcal{G}(4c+1) \oplus \mathcal{G}(4c) = 2 =  \mathcal{G}(6c+1) \oplus \mathcal{G}(2c)   \]
\[ \mathcal{G}(4c+1) \oplus \mathcal{G}(6c) = 5 =  \mathcal{G}(8c+1) \oplus \mathcal{G}(2c)  \]
\[ \mathcal{G}(4c+1) \oplus \mathcal{G}(12c)  = 6 =  \mathcal{G}(10c) \oplus \mathcal{G}(6c+1)\]
Otherwise $(h_2, h_3, h_4)$ must be $(2c, 8c, 10c)$. In this case, apply
\[ \mathcal{G}(4c+1) \oplus \mathcal{G}(2c)  \oplus \mathcal{G}(8c) \oplus \mathcal{G}(10c) 
= 7 =  \mathcal{G}(6c+1) \oplus \mathcal{G}(6c) \oplus \mathcal{G}(6c) \oplus \mathcal{G}(6c)
. \]

\textbf{Case 3:} Exactly two or three of $h_1, h_2, h_3, h_4$ are $h_i = 4c+1$.

Apply \eqref{ob3.2}.
 
\textbf{Case 4:} $h_1=h_2=h_3=h_4=4c+1$.

Apply \eqref{ob3.1}.
\end{proof}

Applying the same technique, we find that, for $p\geq 4$, these subsequences really only begin at entering partitions which have a part of size $1$.

\begin{lem}\label{entering lemma}
Assume Theorem \ref{main target} holds up to $n$.
Suppose $p \geq 4$ and $\mathcal{C} = \{1,2c\}, c\geq 2.$  Then each entering partition in $\mathcal{N}(n,p,\mathcal{C})$ without a part of size $1$ has the same nim-value as some non-entering partition in $\mathcal{N}(n,p,\mathcal{C})$.
\end{lem}

\begin{proof}
Suppose $O_{n}=(h_1,h_2,h_3,h_4,\ldots, h_p)$ is an entering partition so that $h_1,h_2,h_3,h_4>1$.  We will find a non-entering partition $O'_n=(h'_1,h'_2,h'_3,h'_4,\ldots, h'_p)$ with the same nim-value.  As in the previous proof, we will only alter the remainders of $h_1,h_2,h_3,h_4$ modulo $12c$, so Proposition \ref{prop1} assures us that we only need to check that the nim-sum of the nim-values of these remainders remains unaltered.
The seven innumbers less than or equal to $12c$ are precisely $1, 2c+1, 4c+1, 4c+2, 6c+1, 8c+1, 10c+1$.  In some of the cases below, we replace $h_i$ with $h_i - 1$.  So even though it may be that $h_i=1\bmod 2c$, we do require $h_i>1$.

\textbf{Case 1:} $h_1=h_2$ and $h_3=h_4$.  

Apply \eqref{ob3.1}.

\textbf{Case 2:} $h_1=h_2$  and $h_3 \neq h_4$.  

Then one of $h_3$ or $h_4$ is not $4c+1$, say $h_3$.  Apply \eqref{ob3.3} if $h_3=1\bmod 2c$  or \eqref{ob3.4}  if $h_3=4c+2$.

\textbf{Case 3:}  $h_1, h_2, h_3, h_4$ are all different. 

\textbf{Case 3.1:} All $h_i$ are $1 \bmod 2c$.  We will show by subcases how to find an entering partition with the same nim-value that has repeated values.  Then Case 1 or Case 2 applies.

\textbf{Case3.1a:} $\{h_1,h_2\}=\{8c+1,10c+1\}$ then $h_3\in\{1,2c+1,4c+1,6c+1\}$.  We will be able to apply one of the following to obtain a repeated value of either $8c+1$ or $10c+1$:
\begin{align*}
\mathcal{G}(10c+1) \oplus \mathcal{G}(1) &= 6 = \mathcal{G}(8c+1) \oplus \mathcal{G}(2c+1)  \\
\mathcal{G}(10c+1) \oplus \mathcal{G}(4c+1) &= 7 = \mathcal{G}(8c+1) \oplus \mathcal{G}(6c+1) 
\end{align*}

\textbf{Case3.1b:} $h_1,h_2,h_3\in\{1,2c+1,4c+1,6c+1\}$.  Then apply:
\[\mathcal{G}(6c+1) \oplus \mathcal{G}(1) = 3 = \mathcal{G}(2c+1) \oplus \mathcal{G}(4c+1)\]

\textbf{Case 3.2:} $h_1=4c+2$.

\textbf{Case 3.2a:} $h_1=4c+2$, $h_2=4c+1$.  Apply 
\[\mathcal{G}(4c+2) \oplus \mathcal{G}(4c+1) = 5 = \mathcal{G}(8c+2) \oplus \mathcal{G}(1)\]

\textbf{Case 3.2b:} $h_1=4c+2$, $h_2=6c+1$.  Apply 
\[\mathcal{G}(4c+2) \oplus \mathcal{G}(6c+1) = 7 = \mathcal{G}(8c+2) \oplus \mathcal{G}(2c+1)\]

\textbf{Case 3.2c:} $(h_1,h_2, h_3, h_4) = (4c+2, 1, 8c+1, 10c+1)$ or $(h_1,h_2, h_3, h_4) = (4c+2, 2c+1, 8c+1, 10c+1)$.

Apply one of the calculations in Case 3.1a to obtain a repeated value of $8c+1$.

\textbf{Case 3.2d:}  $(h_1,h_2, h_3, h_4) = (4c+2, 1, 2c+1, 10c+1)$.  Apply
\[\mathcal{G}(4c+2) \oplus \mathcal{G}(1) \oplus \mathcal{G}(2c+1) \oplus \mathcal{G}(10c+1)= 0 = \mathcal{G}(10c) \oplus \mathcal{G}(6c)\oplus \mathcal{G}(4)\oplus \mathcal{G}(1)\]

\textbf{Case 3.2e:}  $(h_1,h_2, h_3, h_4) = (4c+2, 1, 2c+1, 8c+1)$.  Apply 
\[\mathcal{G}(4c+2) \oplus \mathcal{G}(1) \oplus \mathcal{G}(2c+1) \oplus \mathcal{G}(8c+1)= 2 = \mathcal{G}(8c) \oplus \mathcal{G}(3c+1)\oplus \mathcal{G}(3c+1)\oplus \mathcal{G}(3)\]
%
\end{proof}
Putting together everything we now know about entering and exiting partitions, we obtain a remarkable recursion on the nim-sets.

\begin{cor} \label{two}
Assume Theorem \ref{main target} holds up to $n$. For  $p \geq 4$ and $c \geq 2$,
\[ \mathcal{N}(n+1, p+1, \{1,2c\}) = \mathcal{N}(n, p, \{1,2c\}).\]
\end{cor}

\begin{proof}  

Fix $p \geq 4$.  We proceed by induction on $n$.  The base case is $n=p$, and we have
\[ \mathcal{N}(p, p, \{1,2c\})=\{\mathcal{G}((1_1,1_2,\ldots,1_p))\}=\{0\}=\{\mathcal{G}((1_1,1_2,\ldots,1_{p+1}))\}=\mathcal{N}(p+1, p+1, \{1,2c\}) \]

Next, fix $n>p$, and suppose $\mathcal{N}(n-1, p, \{1,2c\}) =  \mathcal{N}(n, p+1, \{1,2c\})$. Consider any $O_n$ with $p$ parts.  Appending a pile of size $1$ does not change the nim-value of an option, so we can define $O'_{n+1}$ with $p+1$ parts accordingly.  Hence $\mathcal{N}(n, p, \{1,2c\}) \subset  \mathcal{N}(n+1, p+1, \{1,2c\})$.

Next suppose $O_{n+1, p+1}$ is a partition of $n+1$ with $p+1$ parts.  If one of the parts is precisely $1$, then removing this part gives us a partition of $n$ with $p$ parts that has the same nim-value.  

Next assume $O_{n+1, p+1}$ has no parts of size $1$.  If it is an entering partition, apply Lemma \ref{entering lemma} to obtain a non-entering partition $O'_{n+1,p+1}$ of $n+1$ with $p+1$ parts so that $\mathcal{G}(O'_{n+1,p+1})=\mathcal{G}(O_{n+1,p+1})$.  Then apply Lemma \ref{rem1} to obtain a partition $O_{n,p+1}$ of $n$ with $p+1$ parts so that $\mathcal{G}(O'_{n+1,p+1})=\mathcal{G}(O_{n,p+1})\oplus 1$.  Then by the induction hypothesis, there is some partition $O_{n-1,p}$ of $n-1$ with $p$ parts satisfying $\mathcal{G}(O_{n-1,p})=\mathcal{G}(O_{n,p+1})$.  Next if $O_{n-1,p}$ is an exiting partition, apply Lemma \ref{one} to replace it with a non-exiting partition $O'_{n-1,p}$ with the same nim-value (and $p$ parts).  Then apply Lemma \ref{rem1} to obtain a partition $O_{n,p}$ of $n$ with $p$ parts satisfying $\mathcal{G}(O_{n,p})=\mathcal{G}(O'_{n-1,p})\oplus 1$. 
We obtain
\[ \mathcal{G}(O_{n+1,p+1})=\mathcal{G}(O'_{n+1,p+1})=\mathcal{G}(O_{n,p+1})\oplus 1=\mathcal{G}(O_{n-1,p})\oplus 1=\mathcal{G}(O'_{n-1,p})\oplus 1=\mathcal{G}(O_{n,p}), \]
And hence $\mathcal{N}(n+1, p+1, \{1,2c\})  \subset  \mathcal{N}(n, p, \{1,2c\})$.
\end{proof}

\section{Proof of the main result}

As mentioned above, \cite{par} verified Theorem \ref{main target} for $c=2,3,4,5$.   Our strategy is to show that for $c>3$, the nim-sequence for $\mathcal{C}=\{1,2c\}$ is a sort of `shifted-expanded version' of the nim-sequence for $\mathcal{C}=\{1,6\}$.  In fact, we will derive this from the same correspondence on the associated nim-sets.  Looking just at the statement of the theorem, we might think that there are a few ways we might define the correspondence.  However, when we look at a few examples of individual nim-sets, we see that our choices are limited.
Let's compare the Maple generated  nim-sets $\mathcal{N}(n, 3, \{1,6\})$ for $1\leq n\leq 14$:
\[ \{\;\}, \{\;\}, \{0\}, \{1\}, \{0\}, \{1\}, \{0\}, \{1\}, \{0, 2\}, \{1, 3\}, \{0, 2\}, \{1, 3\}, \{0, 2\}, \{1, 3\} \]
and $\mathcal{N}(n, 3, \{1,10\})$ for $1\leq n\leq 22$:
\[\{\;\}, \{\;\}, \{0\}, \{1\}, \{0\}, \{1\}, \{0\}, \{1\}, \{0\}, \{1\}, \{0\}, \{1\},\] 
\[\{0, 2\}, \{1, 3\}, \{0, 2\}, \{1, 3\}, \{0, 2\}, \{1, 3\}, \{0, 2\}, \{1, 3\}, \{0, 2\}, \{1, 3\}\]

That the first $2$ (i.e., $p-1$) terms are empty makes sense, since we cannot split a pile of size $< p$ into $p$ piles.  After that, we have the alternating subsequences that we saw earlier, in blocks of length $2c$.  However, there is a bit more subtlety elsewhere.  Let's compare $\mathcal{N}(n, 3, \{1,6\})$ for $15\leq n\leq 20$:
\[\{0, 1, 2\}, \{0, 1, 3, 4\}, \{0, 1, 2, 5\}, \{0, 1, 3, 4\}, \{0, 1, 2, 5\}, \{0, 1, 3, 4\}\]
with $\mathcal{N}(n, 3, \{1,10\})$ for $23\leq n\leq 32$:
  \[\{0, 1, 2\}, \{0, 1, 3, 4\}, \{0, 1, 2, 5\}, \{0, 1, 3, 4\}, \{0, 1, 2, 5\},\{0, 1, 3, 4\}, \]
  \[\{0, 1, 2, 5\}, \{0, 1, 3, 4\}, \{0, 1, 2, 5\}, \{0, 1, 3, 4\}\]
Here we see that the terms in position $2cq+1+(p-1)$ might be irregular.  Next compare $\mathcal{N}(n, 2, \{1,6\})$ for $21\leq n\leq 26$:
\[\{1, 2, 4, 6\}, \{0, 3, 5, 7\}, \{1, 2, 4, 6\}, \{0, 3, 5, 7\}, \{2, 6\}, \{0, 1, 3, 4, 7\}\]
with $\mathcal{N}(n, 2, \{1,10\})$ for $33\leq n\leq 42$:
\[\{1, 2, 4, 6\}, \{0, 3, 5, 7\}, \{1, 2, 4, 6\}, \{0, 3, 5, 7\},\]
\[\{1, 2, 4, 6\}, \{0, 3, 5, 7\}, \{1, 2, 4, 6\}, \{0, 3, 5, 7\},\{2, 6\}, \{0, 1, 3, 4, 7\}\]
Here we see that the terms in positions $2cq-1 + (p-1)$ and $2cq + (p-1)$ might be irregular.  Putting this all together, we define our corresondence as follows:

\begin{defi}
Given $1 \leq r \leq 2c$, define 
\[r'=
\begin{cases}
1&\text{ if }r=1\\
2&\text{ if }r=2\\
3&\text{ if }2<r<2c-1\text{ and }r\text{ is odd}\\
4&\text{ if }2<r<2c-1\text{ and }r\text{ is even}\\
5&\text{ if }r=2c-1\\
6&\text{ if }r=2c
\end{cases}
\]
In other words,
\begin{center}
\begin{tabular}{c | cl c l c l c l c l c ccc c  }
 $r$ & 1 & 2 & 3 & 4 & 5 & 6 & 7 & \dots & $2c-3$ & $2c-2$& $2c-1$ & $2c$ \\ \hline
 $r'$& 1 & 2 & 3 & 4 & 3 & 4 & 3 & \dots & 3 & 4& 5 & 6  
\end{tabular}
\end{center}
Next suppose $p\geq 1$, $c \geq 3$, and $n \geq p$.  Write $n=k+p-1=2cq+r+p-1$  with $1\leq r\leq 2c$.
Then for each $p\geq 1$, $n\geq p$, set 
\[\phi_p(n) = \phi_p(2cq+r+p-1) = 6q+r'+p-1\]
\end{defi}
Since \cite{par} proved Theorem \ref{main target} for $c=2,3$, the following theorem finishes the proof of Theorem \ref{main target}.
\begin{thm}\label{thm5}  For $c\geq 3$
\[ \mathcal{G}_{\{1,2c\}}(n) = \mathcal{G}_{\{1,6\}}(\phi_1(n)).\]
\end{thm}
As we said above, we first prove that the $\phi$ correspondence holds on the corresponding nim-sets.  This is quite surprising, as it is true for all $p$, not just $p=2$ and $p=2c+1$.
\begin{lem}\label{three}
For $c \geq 3$ and $k \geq 1$, $p \geq 2$,
\[  \mathcal{N}(k+p-1, p, \{1,2c\}) = \mathcal{N}(\phi_p(k+p-1),  p, \{1,6\}).\]
\end{lem}

Before we prove this, we need some facts about $\phi$ for our base cases.

\begin{lem}\label{map}Fix $c\geq 3$ and $\mathcal{C}=\{1,2c\}$.
\begin{enumerate}
\item[(a)] For all $p$, $\phi_p$ preseves parity.

\item[(b)] The map $\phi_1$ induces bijections from the innumbers (outnumbers) of $\mathcal{C}=\{1,2c\}$ to the innumbers (outnumbers) of $\mathcal{C}=\{1,6\}$.

\item[(c)]  Suppose $2\leq p\leq 4$ and $h_1,\ldots,h_p$ are all innumbers or 
all outnumbers.  If it is not the case that $h_1=h_2=h_3=h_4=2\bmod 2c$, then 
\[ \phi_1(h_1)+\phi_1(h_2)+\dots+\phi_1(h_p) =  \phi_p(h_1+h_2+\dots+h_p). \]

\item[(d)] Suppose $O_n$ is a partition on $n$ with $p$ parts, and that $\left\lfloor O_n\right\rfloor$ and $\left\lceil O_n\right\rceil$ are partitions of $a$ and $b$, respectively.  If $1\leq p\leq 3$, then $\phi_p(a)\leq \phi_p(n)\leq \phi_p(b)$.   

\end{enumerate}
\end{lem}

\begin{proof}
To prove (a), note that the map $r\mapsto r'$ preserves parity.  

Part (b) is clear from the construction of $\phi_1$.

For (c) and (d), we first examine the effects of $\phi_2$, $\phi_3$, and $\phi_4$ on the remainders of a number mod $2c$.  The full calcuations are in Tables \ref{table1}, \ref{table2}, and \ref{table3}.  Recall that in the definition of $\phi_p$, the remainder of $n$ mod $2c$ is $s=r+p-1$.   The tables all assume that $c\geq 4$, since if $c=3$, then all maps $\phi_p$ are the identity, and the lemma holds trivially.
\begin{table}[ht]
\begin{center}
\begin{tabular}{c | cl c l c l c l c l c ccc}
  $s=r+1$ & 1 & 2 & 3 & 4 & 5 & 6 & 7 & \dots & $2c-3$ & $2c-2$& $2c-1$ & $2c$ \\ \hline
  $r$ & $2c$ & 1 & 2 & 3 & 4 & 5 & 6 & \dots & $2c-4$ & $2c-3$& $2c-2$ & $2c-1$ \\ 
  $r'$ & $6$ & 1 & 2 & 3 & 4 & 3 & 4 & \dots & $4$ & $3$& $4$ & $5$ \\ \hline
  $r'+1$& 1 & 2 & 3 & 4 & 5 & 4 & 5 & \dots & 5 & 4& 5 & 6  
\end{tabular}
\caption{The effect of $\phi_2$ on remainders mod $2c$}
\label{table1}
\end{center}
\end{table}

\begin{table}[ht]
\begin{center}
\begin{tabular}{c | cl c l c l c l c l c ccc}
  $s=r+2$ & 1 & 2 & 3 & 4 & 5 & 6 & 7 & \dots & $2c-3$ & $2c-2$& $2c-1$ & $2c$ \\ \hline
  $r$ & $2c-1$ & $2c$ & 1 & 2 & 3 & 4 & 5 & \dots & $2c-5$ & $2c-4$& $2c-3$ & $2c-2$ \\
  $r'$ & $5$ & $6$ & 3 & 4 & 3 & 4 & 3 & \dots & $3$ & $4$& $3$ & $4$ \\ \hline
  $r'+2$& 1 & 2 & 3 & 4 & 5 & 6 & 5 & \dots & 5 & 6& 5 & 6  
\end{tabular}
\caption{The effect of $\phi_3$ on remainders mod $2c$}
\label{table2}
\end{center}
\end{table}

\begin{table}[ht]
\begin{center}
\begin{tabular}{c | cl c l c l c l c l c ccc}
  $s=r+3$ & 1 & 2 & 3 & 4 & 5 & 6 & 7 & \dots & $2c-3$ & $2c-2$& $2c-1$ & $2c$ \\ \hline
  $r$ & $2c-2$ & $2c-1$ & $2c$ & 1 & 2 & 3 & 4 & \dots & $2c-6$ & $2c-5$& $2c-4$ & $2c-3$ \\ \hline
  $r'$ & $4$ & $5$ & $6$ & 1 & 2 & 3 & 4 & \dots & $4$ & $3$& $4$ & $3$ \\ \hline
  $r'+3$& 1 & 2 & 3 & 4 & 5 & 6 & 7 & \dots & 7 & 6& 7 & 6  
\end{tabular}
\caption{The effect of $\phi_4$ on remainders mod $2c$}
\label{table3}
\end{center}
\end{table}

For (c), suppose that $p=2$ and $h_1=2cq_1+s_1$ and  $h_2=2cq_2+s_2$ are innumbers.  Then $s_i=1\text{ or 2}$.  Thus $s_1+s_2=2,3,\text{ or }4$.  Thus Table \ref{table1} implies that 
\[\phi_2(h_1+h_2)=\phi_2(2c(q_1+q_2)+s_1+s_2)=6(q_1+q_2)+s_1+s_2=(6q_1+s_1)+(6q_2+s_2)=\phi_1(h_1)+\phi_1(h_2).\]
The innumber cases for $p=3$ and $p=4$ are similar, except when  $h_1=h_2=h_3=h_4=2\bmod 2c$.  In this case, $h_1+h_2+h_3+h_4=8\bmod 2c$, but $\phi_4(8)\neq 8$, so the equation fails.  Since outnumbers must be either $2c\text{ or }1\bmod 2c$, an examination of the tables shows that the equation holds in these cases as well.  For example, suppose that $p=2$ and $h_1=2cq_1+2c$ and  $h_2=2cq_2+s_2$ are outnumbers, in particular $s_2=1\text{ or }2c$.  Then we have
\[\phi_2(h_1+h_2)=\phi_2(2c(q_1+q_2+1)+s_2)=6(q_1+q_2+1)+s'_2=(6q_1+6)+(6q_2+s'_2)=\phi_1(h_1)+\phi_1(h_2)\]

Part (d) is immediate for $p=1$, so let $p=2$ and suppose  $O_n=\{2cq_1+s_1,2cq_2+s_2\}$. Then $\left\lfloor 2cq_i+s_i\right\rfloor=2cq+a_i$, where $a_i=1\text{ or }2$, and $\left\lceil 2cq_i+s_i\right\rceil=2cq+b_i$, where $b_i=1\text{ or }2c$.  Then
\begin{align*}
\phi_2(2cq_1+a_1+2cq_2+a_2)&=6q_1+6q_2+a'_1+a'_2\\
&\leq \phi_2(6q_1+6q_2+s_1+s_2)\\
&\leq 6q_1+6q_2+b'_1+b'_2\\
&= \phi_2(6q_1+b_1+6q_2+b_2)
\end{align*}
The equalities are from part (c).  The inequalities are ensured by the values in Table \ref{table1}.  For example since $a_1+a_2\leq 4$, and so $a'_1+a'_2$ is less than or equal to all other values in Table \ref{table1} to the right of it.  

The case for  $p=3$ is similar, except that it is not the case $a'_1+a'_2+a'_3\leq s_1+s_2+s_3$ when $a'_1=a'_2=a'_3=2$ and $s_1+s_2+s_3$ is odd and less than $2c$.  In this case, $s_1+s_2+s_3=5$.  However, this case never occurs, due to the last clause of Definition \ref{definition of floor}.
\end{proof}

To explain the role of Lemma \ref{map}(c) in the proof of Lemma \ref{three}, let's look at an example with $p=2$.
Suppose   $O_n = (2ac+1, 4c+2)$, where $n=2ac+1+4c+2=2c(a+2)+3$. Then
\[ \mathcal{G}_{\{1,2c\}}(O_n) = \mathcal{G}_{\{1,2c\}}(2ac+1) 
\oplus \mathcal{G}_{\{1,2c\}}(4c+2)
= \mathcal{G}_{\{1,6\}}(6a+1) \oplus \mathcal{G}_{\{1,6\}}(12+2)  
= \mathcal{G}_{\{1,6\}}(O'_{n'}). \]  
The first and last equalities are by definition, and the middle one will be by assuming the induction hypothesis on $n$ from Theorem \ref{thm5}.  Lemma \ref{map}(c) then verifies that the resulting option $O'_{n'}$ corresponds to $O_n$ via $\phi_2$.  In other words, that the third equality of the following holds:
\[n'
=(6a+1)+(12+2)
=\phi_1(2ac+1) + \phi_1(4c+2) 
=\phi_2(2ac+1+4c+2)
=\phi_2(n)\]
\begin{proof}[proof of Lemma \ref{three}]  Let $n=k+p-1$. As we saw in our examples above, this holds trivially if $n<p$. Suppose $n\geq p$, and toward an induction that for all $n'<n$, both Lemma \ref{three} and Theorem \ref{thm5} hold.   
\paragraph{\textbf{Case 1: }$2\leq p\leq 3$.}
In order to show $\mathcal{N}(n, p, \{1,2c\}) \subseteq \mathcal{N}(\phi_p(n),  p, \{1,6\})$, let $O_n=(h_1,\ldots,h_p)$ be a partition of $n$ with $p$ parts.  If $O_n$ is 
either an entering or exiting partition, 
then
\begin{align*}
\mathcal{G}_{\{1,2c\}}(O_n)&=\mathcal{G}_{\{1,2c\}}(h_1)\oplus\ldots\oplus\mathcal{G}_{\{1,2c\}}(h_p)&\text{Definition \ref{nim value of an option}}\\
&=\mathcal{G}_{\{1,6\}}(\phi_1(h_1))\oplus\ldots\oplus\mathcal{G}_{\{1,6\}}(\phi_1(h_p))&\text{induction on $n$}\\
&=\mathcal{G}_{\{1,6\}}(\phi_1(h_1),\ldots,\phi_1(h_p))&\text{Definition \ref{nim value of an option}}\\
&\in  \mathcal{N}(\phi_1(h_1)+\ldots+\phi_1(h_p),  p, \{1,6\})&\text{Definition \ref{def-nim set}}\\
&= \mathcal{N}(\phi_p(h_1+\ldots+h_p),  p, \{1,6\})&\text{Lemma \ref{map}(c)}\\
&= \mathcal{N}(\phi_p(n),  p, \{1,6\})&\text{}\\
\end{align*}
If $O_n$ is an intermediate parition, then consider its floor $\left\lfloor O_n\right\rfloor$ and ceiling $\left\lceil O_n\right\rceil$, partitions on $a,b$ respectively, with $a<n<b$.  By the above calculation,  
$\mathcal{G}_{\{1,2c\}}(\left\lfloor O_n\right\rfloor)\in \mathcal{N}(\phi_p(a),  p, \{1,6\})$,
$\mathcal{G}_{\{1,2c\}}(\left\lceil O_n\right\rceil)\in \mathcal{N}(\phi_p(b),  p, \{1,6\})$, and $\mathcal{G}_{\{1,2c\}}(O_n)\in  \mathcal{N}(\phi_1(h_1)+\ldots+\phi_1(h_p),  p, \{1,6\})$.  Furthermore, 
\[ \phi_p(a)=
\phi_1(\left\lfloor h_1\right\rfloor)+\ldots+\phi_1(\left\lfloor h_p\right\rfloor) 
\leq \phi_1(h_1)+\ldots+\phi_1(h_p) 
\leq \phi_1(\left\lceil h_1\right\rceil)+\ldots+\phi_1(\left\lceil h_p\right\rceil)
=\phi_p(b) \] 
where the inequalities are Lemma \ref{map}(d) and the equalities are Lemma \ref{map}(c).  Let $Q$ be a partition of $\phi_1(h_1)+\ldots+\phi_1(h_p)$ with $p$ parts satisfying $\mathcal{G}_{\{1,6\}}(Q)=\mathcal{G}_{\{1,2c\}}(O_n)$.
By Lemma \ref{map}(a), $\phi_p(n)$ is the same parity as $\phi_1(h_1)+\ldots+\phi_1(h_p)$, and Lemma \ref{map}(d) implies $\phi_p(a)\leq\phi_p(n)\leq\phi_p(b)$,  so by Corollary \ref{height calculation}, there is a partition $R$ of $\phi_p(n)$ with $p$ parts satisfying $\mathcal{G}_{\{1,6\}}(R)=\mathcal{G}_{\{1,6\}}(Q)$.  Hence $\mathcal{G}_{\{1,2c\}}(O_n)\in \mathcal{N}(\phi_p(n),  p, \{1,6\})$.

We still need to show $\mathcal{N}(\phi_p(n),  p, \{1,6\}) \subseteq \mathcal{N}(n, p, \{1,2c\})$.  First notice that the map $\phi_p$ is surjective on the set of all partitions.  Thus we will be done if we can show $\phi_p(n)=\phi_p(n')\implies \mathcal{N}(n, p, \{1,2c\})=\mathcal{N}(n', p, \{1,2c\})$.  Examining Tables \ref{table1} and \ref{table2}, we see that we only need to show the following.

\begin{claim}
If $p=2$ and $n\geq 5\bmod 2c$, then $\mathcal{N}(n, p, \{1,2c\})=\mathcal{N}(n-1, p, \{1,2c\})\oplus 1$.  Also, if $p=3$ and $n\geq 6\bmod 2c$, then $\mathcal{N}(n, p, \{1,2c\})=\mathcal{N}(n-1, p, \{1,2c\})\oplus 1$.
\end{claim}

Note that in almost each such case if $O_n$ is a partition of $n$, then $O_n$ has a part of at least size $3\bmod 2c$.  Thus, we can apply Corollary \ref{height calculation} to obtain a partition $O_{n-1}$ of $n-1$ where $\mathcal{G}_{1,2c}(O_n)=\mathcal{G}_{1,2c}(O_{n-1})\oplus 1$.  The exceptional case is if the remainders of $O_n$ are precisely $\{2,2,2\}$.  If one of these is not an innumber, then we can again apply Lemma \ref{rem1}.  Otherwise, suppose $O_n=(12cq_1+4c+2,12cq_2+4c+2,12cq_3+4c+2)$.  Then we can use $O_{n-1}=(12cq_1+4c+1,12cq_2+4c+1,12cq_3+4c+3)$.
\paragraph{\textbf{Case 2: } $p=4$.}
Corollary \ref{stick} says that once either $m$ or $m\oplus 1$ enters the sequence $\mathcal{N}(n,4,\{1, 2c\})$ (or $\mathcal{N}(n,4,\{1, 6\})$), then the alternating subsequence $(m,m\oplus1)$ remains in the sequence.  Note that the alternating subsequence has the potential to enter twice: first as $m$ at an odd $n$ or $m\oplus1$ at an even $n$, and second as $m$ at an even $n$ or $m\oplus1$ at an odd $n$.
Thus, since $\phi_4$ preserves parity, we just need to show that each nim value enters  $\mathcal{N}(n,4,\{1, 2c\})$ at $n=k$ if and only if it enters  $\mathcal{N}(n,4,\{1, 6\})$ at $n=\phi_4(k)$.
 
Next Lemmas \ref{entering lemma} and \ref{rem1} imply that we only need to check the values of entering partitions with a part of size $1$.  Lemma \ref{map}(b) implies that $\phi_1$ induces a bijection from entering partitions with a part of size $1$ to entering partitions with a part of size $1$.

So suppose that $O_n=\{h_1,h_2,h_3,h_4\}$ is a partition of $n$ with a part of size $1$, and that it is an entering partition with respect to $\mathcal{C}=\{1,2c\}$.

\begin{align*}
\mathcal{G}_{\{1,2c\}}(O_n)&=\mathcal{G}_{\{1,2c\}}(h_1)\oplus\ldots\oplus\mathcal{G}_{\{1,2c\}}(h_4)&\text{Definition \ref{nim value of an option}}\\
&=\mathcal{G}_{\{1,6\}}(\phi_1(h_1))\oplus\ldots\oplus\mathcal{G}_{\{1,6\}}(\phi_1(h_4))&\text{induction on $n$}\\
&=\mathcal{G}_{\{1,6\}}(\phi_1(h_1),\ldots,\phi_1(h_4))&\text{Definition \ref{nim value of an option}}
\end{align*}
But then $\{\phi_1(h_1),\ldots,\phi_1(h_4)\}$ is a partition with a part of size $1$, and is entering with respect to $\mathcal{C}=\{1,6\}$.  Then Lemma \ref{map}(c) says that it is a partition of $\phi_4(n)$.
\paragraph{\textbf{Case 3: }  $p  \geq 5$.}
We proceed by induction on $p$,
using $p=4$ as the base case. Note that Lemma \ref{map}(c) fails for $p\geq 5$, so the above argument does not apply.  Also the following argument would not work above, since Corollary \ref{two} only holds if $p\geq 4$.  Writing $n=k+p-1=2cq+r+p-1$, we have
\begin{align*}
\mathcal{N}(2cq+r+p-1, p, \{1,2c\})
&=  \mathcal{N}(2cq+r+(p-1)-1, p-1, \{1,2c\})&\text{ Corollary \ref{two}}\\
&= \mathcal{N}( \phi_{p-1}(2cq+r+(p-1)-1), p-1, \{1,6\})&\text{ induction on $p$}\\
&= \mathcal{N}( 6q+r'+(p-1)-1, p-1, \{1,6\})&\text{ definition of $\phi_{p-1}$}\\
&= \mathcal{N}( 6q+r'+p-1, p, \{1,6\})&\text{ Corollary \ref{two}}\\
&= \mathcal{N}( \phi_{p-1}(2cq+r+p-1), p, \{1,6\})&\text{ definition of $\phi_{p-1}$}\\
\end{align*}
\end{proof}
Now we can finally prove our main result:
\begin{proof}[Proof of Theorem \ref{thm5}]
As in earlier proofs, we note that Proposition \ref{prop1} allows us to restrict our attention to $1\leq n\leq 12c$. 
\begin{claim}
\begin{enumerate} [(a)]
\item  $ \mathcal{N}(n,2c+1,\{1,2c\}) = \mathcal{N}(\phi_1(n),7,\{1,6\}) $
\item $ \mathcal{N}(n,2,\{1,2c\}) = 
\begin{cases}
  \mathcal{N}(\phi_1(n),2,\{1,6\}) \cup \, \{1\} 
&\text{if } n = 8c+a,\\
&\text{ for some } a \in \{5,7,\dots,2c-3\}\\
 \mathcal{N}(\phi_1(n),2,\{1,6\})  
&\text{otherwise }\\
\end{cases}  $
\end{enumerate}
\end{claim}
Assume the claim.  If it is not the case that $n=8c+a$ for some $ a \in \{5,7,\dots,2c-3\}$, then it is clear that $\mathcal{G}_{\{1,2c\}}(n)=\mathcal{G}_{\{1,6\}}(\phi_1(n))$.

So suppose $a \in \{5,7,\dots,2c-3\}$ and let $n=8c+a$.  Then we can see that 
\[ 1 \in\mathcal{N}( \phi_1(8c+a),7,\{1,6\}) 
= \mathcal{N}(27,7,\{1,6\}) \] 
by considering, for example, the option $O_{27} = (13,3,3,2,2,2,2)$. 
That is $\mathcal{G}_{\{1,6\}}((13,3,3,2,2,2,2)) = 1.$  Thus we  conclude that 
\begin{align*}
 \mathcal{G}_{\{1,2c\}}(n)&=\mex\left\{\mathcal{N}(n,2,\{1,2c\}) \cup \mathcal{N}(n,2c+1,\{1,2c\})\right\}\\
 &=  \mex\left\{\mathcal{N}( \phi_1(n),2,\{1,6\}) \cup \{1\} \cup \mathcal{N}( \phi_1(n),7,\{1,6\})\right\} \\
 &= \mex\left\{\mathcal{N}( \phi_1(n),2,\{1,6\}) \cup \mathcal{N}( \phi_1(n),7,\{1,6\})\right\}\\
 &=\mathcal{G}_{\{1,6\}}(\phi_1(n))
\end{align*}

\begin{proof}[Proof of claim] 
To prove (a), first note that if $n<2c+1$, then $\phi_1(n)<7$, so we have $\mathcal{N}(n,2c+1,\{1,2c\}) = \{\;\} = \mathcal{N}(\phi_1(n),7,\{1,6\}) $.  So suppose $n=2cq+2c+r$ with $1\leq r\leq 2c$ and $0\leq q\leq 4$.  Then
\begin{align*}
\mathcal{N}(n, 2c+1, \{1,2c\})&=\mathcal{N}(2cq+r+2c, 2c+1, \{1,2c\})\\ 
&= \mathcal{N}(2cq+ r+6,7,\{1,2c\}) &\text{Corollary }\ref{two}\\
&=\mathcal{N}(6q+r'+6,7,\{1,6\}) &\text{Lemma }\ref{three}\\
&=\mathcal{N}(\phi_1(n),7,\{1,6\}) &
\end{align*}
For (b), again $\mathcal{N}(1,2,\{1,2c\}) = \{\;\} = \mathcal{N}(\phi_1(1),2,\{1,6\}) $, so let $n=2cq+r+1$ with $1\leq r\leq 2c$ and $0\leq q\leq 5$.  Then we have
\begin{align*}
\mathcal{N}(n, 2, \{1,2c\})&=\mathcal{N}(2cq+r+1, 2, \{1,2c\})\\ 
&=\mathcal{N}(6q+r'+1,2,\{1,6\}) &\text{Lemma }\ref{three}
\end{align*}
If we examine Table \ref{table1}, we see that $6q+r'+1=\phi_1(n)$, except when $r+1\in\{5,7,\ldots,2c-3\}$.  In these cases, $\phi_1(n)=6q+3$, but $6q+r'+1=6q+5$.  As it happens, we can use the computer to compute $\mathcal{N}(6q+3,2,\{1,6\})$ and $\mathcal{N}(6q+5,2,\{1,6\})$ for all $0\leq q\leq 5$.  They are:
\begin{align*}
&q&&0&&1&&2&&3&&4&&5\\
\hline
&\mathcal{N}(6q+3,2,\{1,6\})&&\{1\}&&\{1,3\}&&\{0,1,3,4\}&&\{1,2,4,6\}&&\{0,2,5,6\}&&\{0,1,3,5,7\}\\
&\mathcal{N}(6q+5,2,\{1,6\})&&\{1\}&&\{1,3\}&&\{0,1,3,4\}&&\{1,2,4,6\}&&\{0,1,2,5,6\}&&\{0,1,3,5,7\}
\end{align*}
We see that they are equal except that when $q=4$, $\mathcal{N}(6q+5,2,\{1,6\})=\mathcal{N}(6q+3,2,\{1,6\})\cup\{1\}$, proving the claim.
\end{proof}
\end{proof}

\section{An extension}
After another fact about our nim-sets, we compute the nim-sequence for any version of \textsc{cut} whose cut-set consists of $1$ and even numbers greater than $2$.

\begin{lem} \label{seven}
For $p \geq 4, c \geq 2,$
\[ \mathcal{N}( n, p+2,\{1,2c\}) \subset \mathcal{N}(n,p,\{1,2c\}).  \]
\end{lem}
\begin{proof}
\begin{align*}
\mathcal{N}(n,p+2,\{1,2c\}) &=  \mathcal{N}(n-2,p,\{1,2c\})&\text{Corollary \ref{two} twice}\\
&\subset \mathcal{N}(n,p,\{1,2c\})&\text{Corollary \ref{stick} twice}
\end{align*}
\end{proof}
\begin{thm} \label{eight}
Let $\mathcal{C}_1 = \{1, 2c_1, 2c_2, 2c_3, ...\},   \;\ c_1,c_2, ... \geq 2$
and $\mathcal{C}_2 = \{1,2c\}$ where $c = \min\{c_1,c_2,c_3,\dots\}.$
 Then for $n\geq 1$,
 \[  \mathcal{G}_{\mathcal{C}_1}(n) =  \mathcal{G}_{\mathcal{C}_2}(n). \]
 \end{thm}
 
 \begin{proof}
 We proceed by induction on $n.$ For the base case, it is clear that 
 $\mathcal{G}_{\mathcal{C}_1}(1) =  \mathcal{G}_{\mathcal{C}_2}(1) = 0.$
 For the induction step, we assume the statement is true for all $n'<n$.  This means that for any $p>1$, $\mathcal{N}(n,p,\mathcal{C}_1) =\mathcal{N}(n,p,\mathcal{C}_2)$.  Hence
 \begin{align*}
 \mathcal{G}_{\mathcal{C}_1}(n) &= \mex \left\{ 
 \mathcal{N}(n,2,\mathcal{C}_1) \cup \bigcup_{2c_i\in\mathcal{C}_1}
 \mathcal{N}(n,2c_i+1,\mathcal{C}_1)  \right\} \\
  &= \mex \left\{ 
 \mathcal{N}(n,2,\mathcal{C}_2) \cup \bigcup_{2c_i\in\mathcal{C}_1} \mathcal{N}(n,2c_i+1,\mathcal{C}_2)  \right\}\\
  &= \mex \{ \mathcal{N}(n,2,\mathcal{C}_2) \cup \mathcal{N}(n,2c+1,\mathcal{C}_2)  \}&\text{ Lemma \ref{seven}}\\
  &=  \mathcal{G}_{\mathcal{C}_2}(n)
 \end{align*}
\end{proof}

\section{What's left? }
In this section, we categorize the families of cut sets for which the nim-sequence of \textsc{cut} remains unknown.  There are $4$ such families.
Let $X$ to be a non-empty set of even numbers, 
each of which is at least 4. Let $Y$ to be a non-empty set of odd numbers, each of which is at least 5.  Let $x=2c$ and $y$ be the smallest elements of $X$ and $Y$, respectively.

\textbf{Family A: }$\mathcal{C} =\{1, 3\} \, \cup \, X,$ or $\mathcal{C} =\{1, 3\} \, \cup \, X \, \cup Y$

We already know from \cite[Propositions 8]{par} that  the nim-sequence for $\mathcal{C} =\{1, 3,2c\}$ is $(0,1)^c (+2)$.

\begin{conj}
The nim-sequence for all games of \textsc{cut} in this family are all precisely $(0,1)^c (+2)$.
\end{conj}

\textbf{Family B:}   $\mathcal{C} =\{1\} \, \cup \, X \, \cup Y$

The nim-sequence of this family seems to have some resemblance to
the nim-sequence for $\mathcal{C} =\{1,2c\}$ when $c\geq 2$, but we cannot make a full conjecture at this time.  The following partial extention of Theorem \ref{eight} seems to be true.

\begin{conj}
If $3x < y$, then $\mathcal{G}_{\mathcal{C}}(n)  = \mathcal{G}_{\{1,x\}}(n)$
for $n \geq 1$.
\end{conj}

We note that proving the arithmetic-periodicity of Families A and B would imply Conjecture 1 of \cite{par}.

\textbf{Family C:} $\{1, 2\}\subseteq\mathcal{C}, 3\notin\mathcal{C}, \mathcal{C}\neq\{1, 2\}$.

It is not so clear how to categorize the patterns of this family.  However, we do observe:
\begin{conj}
The nim-sequence for all games of \textsc{cut} in Family C are all \emph{ultimately} arithmetic-periodic.
\end{conj}

\textbf{Family D:}   $\mathcal{C} =\{1, 2\}$

The first 36 terms of the nim- sequence for this game of \textsc{cut} are 
\begin{align*}
&&0, &&1, &&2, &&3, &&1, &&4, &&3, &&2, &&4, &&5, &&6, &&7,
&&8, &&9, &&7, &&6, &&9, &&8,\\
&&11, &&10, &&12, &&13, &&10, &&11,
&&13, &&12, &&15, &&14, &&16, &&17, &&5, &&15, &&17, &&16, &&19, &&18
\end{align*}
It is supposed that the nim-sequence for this particular version of \textsc{cut} is the most difficult to analyze.  We also can not find any pattern here.  In \cite{par}, it was shown that this game is equivalent to the take-and-break game with hexadecimal code 0.7F.


\begin{thebibliography}{CCFM}

\bibitem[LIP]{LIP} M. Albert, R. Nowakowski, and D. Wolfe, 
{\em{Lessons In Play, An Introduction to Combinatorial 
Game Theory}}, A K Peters, Ltd., 2007.

\bibitem[WW1]{WW1} E. Berlekamp, J. H. Conway, and
R. Guy, {\em{Winning Ways for your Mathematical Plays}}, Academic
Press, New York, 1982. 

\bibitem[DDLP]{par}
A. Dailly, E. Duchene, U. Larsson, G. Paris, {\it Partition Games}, 
{Discrete Applied Mathematics}  {{\bf 285} (2020)}, 509--525.


\end{thebibliography}
\end{document}